\documentclass[a4paper,11pt]{article}
\usepackage[utf8]{inputenc}
\usepackage{amsmath,amsfonts,amsthm,amssymb}
\usepackage[margin=1.3in]{geometry}
\usepackage{breqn,fancyref}
\usepackage[hidelinks]{hyperref}
\usepackage{enumerate}
\allowdisplaybreaks

\newcommand{\drugapochodna}[3]{\frac{\partial^2 #1}{\partial #2 \partial #3}}
\newcommand{\poczasie}[2]{\frac{d^{#2} #1}{dt^{#2}}}
\newcommand{\aumlaut}[0]{\"a}
\newcommand{\Kahler}{K\aumlaut hler }
\newcommand{\calkowite}{\mathbb{Z}}
\newcommand{\zespolone}[0]{\mathbb{C}}
\newcommand{\rzeczywiste}[0]{\mathbb{R}}
\newcommand{\torus}{\mathbf{T}}

\newcommand{\email}{\href{mailto:szymon.myga@im.uj.edu.pl}{\nolinkurl{szymon.myga@im.uj.edu.pl}}}
\newcommand{\Monge}{Monge-Amp\`{e}re }
\newcommand{\moment}{\mathbf{m}}

\newtheorem{Thm}{Theorem}[section]
\newtheorem{Prop}[Thm]{Proposition}

\newtheorem{Cor}[Thm]{Corollary}
\newtheorem{Fac}[Thm]{Fact}

\newtheorem*{Lem*}{Lemma}

\theoremstyle{definition}
\newtheorem*{Def}{Definition}
\newtheorem*{Rem}{Remark}

\title{On the conditional measures on the orbits of the complex torus}
\author{Szymon Myga}
\date{}

\begin{document}

\maketitle

\begin{abstract}
 We explore the structure of invariant measures on compact \Kahler manifolds with Hamiltonian torus actions. We derive the formula for conditional measures on the orbits of the complex torus and use it to prove a conditional statement about uniqueness of solutions to the $g$-\Monge equation.
\end{abstract}

\section{Introduction}

Let $(X,\omega)$ be a compact K\aumlaut hler manifold of real dimension $2n$, i.e. $X$ is a complex manifold and one can find a hermitian metric on it whose fundamental form $\omega$ is closed, thus making $(X,\omega)$ into a symplectic manifold. We assume that there is a real torus $\mathbf{T}^k$ acting smoothly and effectively on it by symplectomorphisms. 
This action can be extended to the holomorphic action of complex torus $\mathbf{T}_\zespolone^k \simeq (\zespolone^*)^k$. 
Finally, we assume that the action is Hamiltonian, so there is a momentum map: an action invariant map $\mathbf{m}:X \rightarrow (\rzeczywiste^k)^*$ such that for every element $t$ of $\mathfrak{t}~\simeq \rzeczywiste^k$ -- the Lie algebra of $\mathbf{T}^k$, we have that
\[
 -d\langle \mathbf{m}(p),t\rangle = \omega_p(t^{\#},\cdot), 
\]
with $t^{\#}$ being the vector field generated by $t$ and $\langle \mathbf{m}(p),t\rangle$ the value of the linear form $\mathbf{m}(p)$ at $t$.

By the classical result of Atiyah~\cite{Atiyah}, Guillemin and Sternberg~\cite{GuilleminSternberg} the image $\moment(X) \subseteq \rzeczywiste^k$ is a compact convex polytope. Let us denote it by $\Delta$ and mention that it is an invariant of the action and cohomology group of $\omega$.

Suppose we take an invariant quasiplurisubharmonic function $\phi$ and look at a measure
\[
 \mu := (\omega + i\partial\bar{\partial}\phi)^n.
\]
We normalize $\omega$, so that $\mu$ is a probability measure. There is an open and dense subset $X^0 \subseteq X$ that is a principal $\mathbf{T}_\zespolone^k$-fibration over some manifold $W^0$. For now we assume that $\phi$ is $C^2$ and $\omega + i\partial\bar{\partial}\phi > 0$, hence $\mu$ has a positive density. We look at the orbits of complexified action, which coincide with fibers of $X^0$. In appropriate coordinates the potential of the form $\omega + i\partial\bar{\partial}\phi$ is convex along the orbits of the complex torus. Take $m := n-k$ and let $(x,y,w)$ be those coordinates, namely $(x,y) \in \rzeczywiste^k \times (\rzeczywiste/\calkowite)^k$ and $w \in D \subseteq \zespolone^{m}$. Let $u$ be the potential that locally satisfies $i\partial\bar{\partial}u = \omega + i\partial\bar{\partial}\phi$ and let $\hat{\mu}(w)$ be the density of $\mu$ averaged along the orbits, thus naturally defined on $W^0$ and positive from our assumption. Then we prove the following 

\begin{Thm}
 In the setting sketched above let $\eta_w(x)\,dx$ be the conditional measure of $\mu$ along the orbit $\pi^{-1}(w)$ of the complex torus in $X^0$. It has a density
 \[
   \eta_w(x) = 
  \displaystyle (-1)^{m}\frac{\sigma^{m}(\moment_u,w) }{\hat{\mu}(w)}\det D_x^2 u(x,w),
 \]
 where $\moment_u$ is a momentum map for the symplectic form $\omega + i\partial\bar{\partial}\phi$ and $\sigma^{m}(p,w) = \det D^2_w u^*$ with $p \in \Delta$, is the Hessian of the (locally defined) Legendre transform of $u$, with respect to $w$ variables.
\end{Thm}

\begin{Rem}
 In the case when $k = n$, and hence $X$ is a toric variety, the result is trivial as there is only one open orbit, and the conditional measure is simply $\mu$ restricted to that orbit. That determines $\mu$ completely as it does not put mass on pluripolar sets, so in particular it does not put mass on the singular orbits. For this reason we assume from now on that $k < n$.
\end{Rem}

The Legendre transform $u^*$ in the theorem above is defined locally, but its corresponding $(1,1)$-form in $w$ variables is defined globally and in fact for fixed $p$ is the reduced symplectic form coming from Mardsen-Weinstein reduction. If in turn we fix $w$ and vary $p$, the top form $\sigma^m(p,w)$ becomes the Duistermaat-Heckman measure of the particular $\mathbf{T}_\zespolone^k$-orbit, namely it is a measure on $\Delta$ that comes from transporting $\eta_w(x)\,dx$ through $\moment_u$.

From that perspective we might consider a more general equation for the potentials $\phi$. Instead of $\phi$ solving the \Monge equation, we assume it solves the $g$-\Monge or transport \Monge equation, namely we take a positive continuous function $g$ defined on $\Delta$ and assume that $\phi$ solves
\[\label{eq:monge}
 MA_g(\phi) = g(\moment_\phi)(\omega + i\partial\bar{\partial}\phi)^n = \mu. \tag{*}
\]
This operator was first considered by Berman and Witt Nystr\"{o}m in~\cite{BermanWitt} in relation to \Kahler-Ricci flow and optimal transport problem. In this preprint the authors proved the existence of solutions to the problem and their uniqueness for the measures of finite energy. It was done by adapting the pluripotential methods developed for the variational approach to the complex \Monge equation in~\cite{MongeVariational}.

We also notice that some of the assumptions can be weakened though we still want to use the formula for the conditional measures. For this we only need the positivity of the form on $X^0$. Thus we might assume that $\mu$ is of the form $f\omega^n$, with $f$ positive on $X^0$ but possibly singular on small orbits. Of course there also must be a normalization assumption, i.e. $\int_X f\omega^n = \int_X g(\moment)\omega^n$.

If we consider the densities of the conditional measures with respect to transport \Monge operator the theorem above holds mutatis mutandis with the density of $\eta_w(x)$ being $(-1)^{m}\frac{g(\moment_u)\sigma^{m}(\moment_u,w) }{\hat{\mu}(w)}\det D_x^2 u(x,w)$.

With the formula for conditional measures in hand we can say something about the uniqueness of the transport \Monge equation.

\begin{Thm}
 Suppose we have two solutions to the problem $\eqref{eq:monge}$ with $\mu = f\omega^n$ and $f$ as above, then if they both produce the same momentum map, they are equal up to an additive constant.
\end{Thm}

The equation~\eqref{eq:monge} comes up in the study of \Kahler-Ricci solitons on Fano varieties, which are projective varieties that admit an ample line bundle. A \Kahler-Ricci soliton is a \Kahler form $\omega$ that solves the equation
\[
 \text{Ric}(\omega) - \omega = \mathcal{L}_V\omega,
\]
with $\text{Ric}(\omega)$ being the Ricci form and $\mathcal{L}_V$ the Lie derivative along some holomorphic vector field $V$. By the results of~\cite{BermanWitt} the existence of such a form corresponds to solution of~\eqref{eq:monge} for some function $g_V$.

The upshot of our result is that we do not rely on the variational methods and instead connect the potentials of \Kahler forms with the symplectic geometry of the underlying manifold. That way our result is valid even for the measures of infinite energy. Of course, on the other hand we still have a strong assumption about the positivity on $X^0$.

The structure of the paper is following. In the Section~\ref{sec:torus} we introduce the toric setting in some more detail, in the short Section~\ref{sec:transport} we present the results from the theories of optimal transportation and disintegration od measures that we will use later on. Finally in Section~\ref{sec:wyniki} we present the proofs of the results. Those are rather straightforward as the difficulties lie mainly in making sure that all the tools work together. The last Section is devoted to proving a result about permuting the minors of a matrix.

\section{The torus action and its properties}\label{sec:torus}

In this section we will take a closer look at the symplectic geometry of Hamiltonian torus actions on \Kahler manifolds. The results are classical and can be found in the many excellent textbooks (e.g.~\cite{Audin}, \cite{McDuffSalamon}, \cite{GGK})

We assume that the real torus $\mathbf{T}^k \simeq (S^1)^k$ acts on $(X,\omega)$ as a smooth Lie group in a holomorphic, isometric and effective way. Specifically, there is a map
\[
 \mathbf{T}^k \times X \ni (t,x) \longrightarrow t\cdot x \in X,
\]
smooth in both variables that commutes with the group operations of $\mathbf{T}^k$. We will not make a distinction between the elements of the torus and automorphisms of $X$ generated by those elements. We assume that those automorphisms are in fact symplectomorphisms, thus $t^*\omega = \omega$ for each $t \in \mathbf{T}^k$. The holomorphicity of the action means that the complex structure $J$ of $(X,\omega)$ is preserved, i.e. each $dt$ commutes with $J$. By 'effective' we mean that the action understood as a homomorphism of $\mathbf{T}^k$ into the isometries of $(X, \omega, J)$ is injective.

The action lifts to the mapping from the Lie algebra of the torus $\mathfrak{t}~ (\simeq \rzeczywiste^k)$ to the holomorphic vector fields on $X$. We denote this by $\rzeczywiste^k \ni Y \rightarrow Y^\#$. This vector field can be simply computed at every point from the formula
\[
 Y^\#_x = \left. \poczasie{}{}\right|_{t=0} \exp(tY)\cdot x,
\]
and its called the fundamental vector field associated to $Y$. This map is a monomorphism of algebras where the multiplication of vector fields on $X$ is the usual Lie bracket of vector fields. Since the Lie algebra of a commutative Lie group has a trivial bracket, that implies that the fundamental vector fields of the torus action must commute too. 

We also assume that the action is Hamiltonian, which means that there is a map $\mathbf{m}$ that maps $X$ to the dual algebra $t^*$ in a way that is equivariant and at each $x \in X$ satisfies 
\[
 d\langle \mathbf{m}(x), Y \rangle = -\omega_x(Y^\#,\cdot)
\]
for every $Y \in \rzeczywiste^k$. Since the group is commutative, equivariance here simply means that $\mathbf{m}$ is torus invariant. For fixed $Y$ the function $\langle \mathbf{m}, Y \rangle : X \rightarrow \rzeczywiste$ will be denoted by $\mathbf{m}^Y$ and called the Hamiltonian function of $Y^\#$. Let us finally state one simple property of the momentum map that we will heavily rely on
\begin{Fac}\label{fact}
 The action of $\torus^k$ at $x$ is locally free (i.e. the stabilizer subgroup is finite) if and only if $d\moment_x$ is surjective.
\end{Fac}

By the fundamental results in the theory of Lie group actions, since the action is effective there is an open dense subset where the action is free. Orbits in this subset are called the principal orbits, on the other hand, orbits that have a stabilizer group of positive dimension are called singular. Moreover, since the target manifold is compact, there are only finitely many types of singular orbits and the holomorphicity implies that those orbits form an analytic subset of $X$. Indeed, each such orbit must be a null set of some real holomorphic vector field. Let us denote by $X^0$ the principal orbit subset of $X$, i.e. the subset of $X$ where the action is free and by $X^1$ the complement of singular orbits. That way $X^1$, aside from $X^0$ can also contain the exceptional orbits, i.e. orbits with finite stabilizer subgroups. 

\paragraph{Complexification.} Due to classical Bochner-Montgomery~\cite{BochnerMontgomery} result, the holomorphic action of $\mathbf{T}^k$ can be extended to the holomorphic action of its complexification, although not by isometries anymore. Moreover, this extension can be made somewhat explicit. Since the complexification of $\mathbf{T}^k$ is $(\zespolone^*)^k$ we can represent its elements as $(e^{x_1 + iy_1},\ldots,e^{x_k + iy_k})$ with Lie algebra isomorphic to $\rzeczywiste^k \oplus i\rzeczywiste^k$. Now its a matter of computation to see that the action of elements from $\{y_1 = \ldots = y_k = 0\}$ will induce a vector field of the form $JY^\#$ for $Y$ in $\mathfrak{t}$. Let us denote the imaginary part $i\rzeczywiste^k$ by $H$.

The complexified vector fields can be computed from the momentum mapping. The fundamental observation is the following: let $g$ be the Riemannian metric of $\omega$ and $\nabla_g f := g^{ij}\nabla_j f$ be the metric gradient. Then $\nabla_g \mathbf{m}^Y = JY^\#$. Indeed, let $A$ be any tangent vector, then
\[
 g(\nabla_g \mathbf{m}^Y,A) = d \langle \mathbf{m}, Y\rangle (A) = -\omega(Y^\#,A) = g(JY^\#,A).
\]

\paragraph{Reduction.} In general, producing the space of orbits of a group acting on a manifold by taking the quotient $X/G$ will be of little value. One can readily see this by considering the action of $\rzeczywiste$ on the two dimensional real torus $\mathbf{T}^2$. The line acts by translations on $\rzeczywiste^2$ with $t\cdot (x,y) = (x +t, y +\alpha t)$ and $\alpha$ any irrational number. This action projects to a free action on torus with each orbit being a dense subset, hence the orbit space will fail to even be Hausdorff. On the other hand there is the Quotient Manifold Theorem, which says that the orbit space of a free and proper action is a manifold. Unfortunately, there is no hope for this to hold in the setting described above as the action of the torus must have fixed points. It turns out though that in the case of Hamiltonian actions one can get relatively close to the orbit space being a manifold, moreover the resulting space will also inherit the symplectic and complex structures of original manifold.

\begin{Thm}[Mardsen-Weinstein reduction~\cite{Redukcja}]
 Suppose $(M,\omega)$ is a symplectic manifold with Hamiltonian action of compact group $G$ and a momentum map $\mathbf{m}$. Suppose 0 is a regular value of $\mathbf{m}$ and $G$ acts freely on $\mathbf{m}^{-1}(0)$ then:
 \begin{enumerate} \vspace{-1em}
  \item $\displaystyle M_{red} := \mathbf{m}^{-1}(0)/G$ is a manifold called the reduced space, 
  \item $\displaystyle \pi : \mathbf{m}^{-1}(0) \rightarrow M_{red}$ is a principal $G$-fibration,
  \item $M_{red}$ inherits the symplectic structure of $M$, namely there is a symplectic structure $\sigma$ on $M_{red}$ such that
  \[
   i^*\omega = \pi^*\sigma,
  \]
  where $i : \mathbf{m}^{-1}(0) \rightarrow M$ is the inclusion map.
 \end{enumerate}

\end{Thm}

\begin{Rem}
 In the case of commutative actions there is nothing special about the null value and the theorem will hold for any regular value of the momentum map. 
\end{Rem}
\vspace{-1em}
Let us pick a level set $Z := \mathbf{m}^{-1}(\xi)$ of some regular value $\xi$. It must lie inside $X^1$, but not necessarily in $X^0$, thus the reduction theorem is not directly applicable. However, since by invariance the action of $\mathbf{T}^k$ on $Z$ is almost free, $Z/\mathbf{T}^k$ will have a strucutre of an orbifold. Indeed, it is the canonical way of producing effective orbifolds~\cite{OrbifoldsBook}. Let us denote this orbifold by $W$. Moreover, if we perform the quotient operation on $Z \cap X^0$, we notice that $W$ has a structure of a symplectic orbifold, with the fundamental form defined outside its orbifold singularities. One can also check by computation that $W$ will also inherit the complex structure of $X$. Thus, we are dealing with a \Kahler orbifold.

Instead of considering the orbits of $\mathbf{T}^k$ one might also look at the action of the complex torus $\mathbf{T}_\zespolone^k$. The point $x \in X$ is called semistable iff the closure of $\mathbf{T}_\zespolone^k$-orbit through $x$ intersects the level set $\moment^{-1}(0)$. The point is called stable if the intersection happens at the point where $d\moment$ is of full rank. The sets of stable and semistable point of $X$ will be denoted by $X^s$ and $X^{ss}$ respectively. Note that in the case of the torus, the choice of specific value is arbitrary and we may replace it by any regular value. Moreover, for any choice of regular value the notions of stable and semistable sets will coincide.

Now we may look at the quotient $X^{ss}/\mathbf{T}_\zespolone^k$. It turns out that $X^{ss}/\mathbf{T}_\zespolone^k \simeq W$, thus outside of an analytic subset $X$ can be thought of as an orbifold principal fiber bundle over $W$, in the sense that it becomes a bona fide principal bundle outside of the singular points of $W$.

The principal fiber bundle structure ensures that there are local trivialisations of $X^0$. Let $(z,w) = (z_1,\ldots,z_k,\allowbreak w_1,\ldots,w_m)$ be local holomorphic coordinates on $X^0$ ($m + k = n$) coming from such a bundle trivialisation. That means that $\mathbf{T}^k_{\zespolone}$ acts on $z$'s by $(t + is)\cdot z_j \rightarrow z_je^{t_j +is_j}$. In particular, changing the coordinates by $z_j = e^{x_j +iy_j}$ we have the complex torus $\zespolone^k / 2\pi\mathbb{Z}^k$ acting on $(x,y)$'s by translations.

Let $\phi$ be a local potential for the form $\omega$ on such a coordinate patch, then
\begin{align*}
 \omega = \frac{i}{2}\left( \sum_{i,j = 1}^{k}\drugapochodna{\phi}{z_i}{\bar{z}_j}dz_i \right. &\wedge d\bar{z}_j + \sum_{i,j = 1}^{i =k , j=m}\drugapochodna{\phi}{z_i}{\bar{w}_j}dz_i\wedge d\bar{w}_j \\
 &+ \left. \sum_{i,j = 1}^{i=m,j=k}\drugapochodna{\phi}{w_i}{\bar{z}_j}dw_i\wedge d\bar{z}_j
   + \sum_{i,j = 1}^{m}\drugapochodna{\phi}{w_i}{\bar{w}_j}dw_i\wedge d\bar{w}_j \right).
\end{align*}
Changing $z$'s to $(x,y)$'s and keeping in mind that $\phi$ is invariant in $y$ the following terms in $\omega$ change
\begin{align*}
 \frac{i}{2}\phi_{z_i\bar{z}_j}dz_i\wedge d\bar{z}_j &= \frac{i}{8}\phi_{x_ix_j}(dx_i\wedge dx_j + idy_i\wedge dx_j - idx_i\wedge dy_j + dy_i\wedge dy_j) \\
 \frac{i}{2}\phi_{z_i\bar{z}_i}dz_i\wedge d\bar{z}_i &= \frac{1}{4}\phi_{x_ix_i}dx_i\wedge dy_i \\
 \frac{i}{2}\phi_{z_i\bar{w}_j}dz_i\wedge d\bar{w}_j &= \frac{i}{4}\phi_{x_i\bar{w}_j}(dx_i\wedge d\bar{w}_j + idy_i\wedge d\bar{w}_j) \\
 \frac{i}{2}\phi_{w_i\bar{z}_j}dw_i\wedge d\bar{z}_j &= \frac{i}{4}\phi_{w_i x_j}(i dy_j\wedge dw_i  - dx_j\wedge dw_i).
\end{align*}

The terms with derivatives in $w$'s only will stay the same and the mixed terms $dx_i\wedge dx_j$ and $dy_i \wedge dy_j$ will cancel with the terms coming from $dz_j\wedge d\bar{z}_i$, thus leaving
\[
 \frac{1}{4}\phi_{x_ix_j}(dx_i\wedge dy_j + dx_j\wedge dy_i).
\]
In those coordinates, the momentum map $\moment$ has the form $\nabla_x\phi$, i.e. the gradient of $\phi$ with respect to $x$ variables only. Indeed, in that case, the real torus acts linearly $y \rightarrow y+t$, so the fundamental vector fields are of the form $a^i\frac{\partial}{\partial y_i}$ for some constants $a^i$. It is enough to check the momentum map condition on some basis of $\rzeczywiste^k$, it might as well be the canonical basis. Thus one needs to prove the following:
\[
 -d \phi_{x_i} = \omega\left(\frac{\partial}{\partial y_i},\cdot\right),
\]
which follows easily form the previous formulas.

Suppose we take another invariant symplectic form $\tilde{\omega}$ in the cohomology class $[\omega]$. The $\partial\bar{\partial}$-lemma implies that there exists a global quasiplurisubharmonic function $\psi$ such that $\tilde{\omega} = \omega + i\partial\bar{\partial}\psi$~\cite[chapter 15]{Moroianu}. We stay in $(x,y,w)$ coordinates and look at the properties of this Hamiltonian action. The moment map will again be expressed locally as $\nabla_x(\phi + \psi)$. Moreover, there exists an equivariant symplectomorphism between the two symplectic forms, let us call it $s$ -- this is a simple consequence of the Moser's stability theorem~\cite{McDuffSalamon} adapted to the equivairnat setting. It is now a matter of calculation to show that in that case the momentum maps are also related by $s$, namely that $\nabla_x(\phi + \psi) = (\nabla_x \phi) \circ s$. That implies that the level sets are equivariantly symplectomorphic, hence they are isomorphic as principal fiber bundles, which in turn implies that $W$ depends only on the cohomology class.

\paragraph{Equivariant cohomology.} We would like to connect the cohomology of $W$ to that of $X$. To this end we need to use the notion of equivariant cohomology. We will not introduce this set of ideas in detail as we only need a handful of results, instead we would like to recommend some references (\cite{GuilleminSternbergdeRham}, \cite{BottEquivariantCohomology},\cite[Appendix C]{GGK}). The only facts from this theory we would like to use will be contained in this paragraph. The inclusion map $i: Z_\xi \rightarrow X$, with $Z_\xi$ being as before a level set of $\moment$ for some regular value $\xi$, induces the equiviariant cohomology mapping
\[
 i^*: H_{\mathbf{T}^k}(X) \longrightarrow H_{\mathbf{T}^k}(Z_\xi).
\]
It turns out that $H_{\mathbf{T}^k}(Z_\xi) = H(W)$, with standard, say de Rham cohomology of $W$. In this setting the de Rham cohomology of $W$ is nothing more than the cohomology of $\mathbf{T}^k$-invariant forms on $Z_\xi$, which in fact is isomorphic to the real cohomology of $W$ as a topological space~\cite{OrbifoldsBook}. We are only really interested in second cohomology groups and in this case $H_{\mathbf{T}^k}^2(X)$ has an interpretation that shall look familiar -- it consists of pairs $(\omega, \Phi)$ of closed invariant two-forms and equivariant smooth functions from $X$ to $\rzeczywiste^k$ such that for each $Y \in \rzeczywiste^k$
\[
 \omega(Y^\#,\cdot \;) = -d\Phi^Y.
\]
In that case, $i^*$ called the Kirwan map, sends an invariant symplectic form $\omega$ on $X$ to its reduction on the level set of appropriate momentum mapping
\[
 \omega \rightarrow \sigma(\xi).
\]
It is not hard to see, that it will be exactly the form given by the Mardsen-Wienstein reduction. As any two invariant symplectic forms in the same de Rham cohomology class are also equivariantly cohomologous, we conclude that the reduction is well defined on the level of cohomology classes.

\paragraph{Compact \Kahler manifolds with torus action.} Let us gather everything together now to describe our setting. First let us state the refinement of the convexity theorem proved by Atiyah in~\cite{Atiyah}.
\begin{Thm}
 Suppose there is a Hamiltonian action of a torus on a compact \Kahler manifold $X$. Let $Y$ be an orbit of $\mathbf{T}^k_\zespolone$ through some point. If $Z_j$ are critical points of $\mathbf{m}$ that intersect $\overline{Y}$ let $c_j = \mathbf{m}(Z_j)$, then
 \begin{enumerate} \vspace{-1em}
  \item $\mathbf{m}(\overline{Y})$ is the convex polytope $\Delta'$ with vertices $c_j$,
  \item for each open face $\zeta$ of $\Delta'$, the preimage $\moment^{-1}(\zeta)\ \cap \overline{Y}$ is a single $\mathbf{T}_\zespolone$ orbit,
  \item $\moment$ induces a homeomorphism of $\overline{Y}/\mathbf{T}^k$ onto $\Delta'$.
 \end{enumerate}
\end{Thm}

\begin{Rem}
 Let us note here that if $\dim X \neq 2k$ then $\Delta'$ might happen to be of the same dimension as the full polytope $\Delta = \moment(X)$ and yet be a proper subset of it.
\end{Rem}

Thus putting everything together we arrive at the following picture. The moment polytope divides into open convex chambers of regular values. The preimages of those chambers are $(\zespolone^*)^k$-principal fiber bundles over \Kahler orbifolds $W$ that possibly differ chamber to chamber. 

Suppose we would like to use the information that is contained on each orbit. Let us restrict ourselves to a single chamber $\Delta'$ in $\Delta$. We can use the momentum map to define the map $\moment^+$ on $X^1$ -- a bundle map between $(\zespolone^*)^k$-bundle over $W$ and $\Delta'$-bundle over the same base. We simply write the map down locally
\[
 \moment^+: (\zespolone^*)^k \times U \ni (z,w) \rightarrow (\moment(z,w),w) \in \Delta \times U,
\]
for any local trivialistion over $U \subseteq W$ and notice that it is well defined by the fact that $\moment$ is a global map. What is worth noting here is that $\moment^+$ 'untangles' the fibration $X^0$ into simple cartesian product.



\section{Some analytic and measure-theoretic preliminaries}\label{sec:transport}

\subsection{Optimal transport and Monge-Amp\`{e}re equation}

We say that the function $T:\rzeczywiste^n \rightarrow \rzeczywiste^n$ transports probability measure $\mu$ to probability measure $\nu$ if for any Borel set $A$ 
the following equality holds
\[
 \nu[A] = \mu[T^{-1}(A)].
\]
Alternatively we say that $T$ pushes $\mu$ forward to $\nu$ and denote the push-forward measure by $T_{\#}\mu$.

In general there will be a lot of such maps, so it is natural to put some optimality constraints on them. The best understood constraint and in some cases the natural one is minimizing the quadratic cost, i.e. the transport map should  minimize the following functional
\[
 \int_{\rzeczywiste^d} |x - T(x)|^2 \,d\mu.
\]
In general there might not be a solution and if it exists it might not be unique, some regularity assumptions for the measures must be added. For example, one can assume that the measures have finite second moments and $\mu$ is absolutely continuous. In that case the solution exists and has a form of $T=\nabla \phi$ for some convex function $\phi$. For thorough discussion of this problem, the reader might consult~\cite{Villani}.

Supposing that a solution exists, by the transport condition we get
\[
 \int\chi_A \,d\nu = \int_A\, d\nu = \int_{(\nabla\phi)^{-1}(A)}\,d\mu = \int \chi_A\circ \nabla\phi\, d\mu.
\]
That can easily be generalized to get that for any $f \in C_b(\rzeczywiste^n)$
\begin{equation}\label{eq:rownanie}
 \int f \,d\nu = \int f\circ\nabla\phi\, d\mu. 
\end{equation}
Here $C_b(\rzeczywiste^n)$ denotes the set of continuous and bounded functions on $\rzeczywiste^n$.

Suppose now that $d\nu = g(x)dx$ for some density $g(x)$ and $\phi$ is a $C^2$ function. By change of variables formula we get that
\[
  \int f(\nabla\phi(x))\, d\mu = \int f(\nabla \phi(x))g(\nabla \phi(x))\det D^2\phi\, dx
\]
and that provides one with a notion of solution to the transported \Monge equation
\[
 MA^{\rzeczywiste}_g(\phi) := g(\nabla \phi(x))\det D^2\phi = \mu
\]
as long as the optimal transport map exists. There is a rich theory of regularity results for this equation. When the solution $\phi$ becomes singular, its gradient $\nabla\phi$ turns into the subgradient $\partial\phi$ defined by
\[
 p \in \partial \phi(x) \quad \text{iff} \quad \forall z, \ \phi(z) \geq \phi(x) + \langle p, z-x \rangle.
\]

As we mentioned, for any two probability measures the optimal transport solution might not exist. However, under a mild regularity assumption it is still possible to transport one to another through a subgradient of convex function, so that the condition~(\ref{eq:rownanie}) is still satisfied. This is the content of the following important theorem. 

\begin{Thm}[McCann~\cite{McCann}]
 Let $\mu, \nu$ be probability measures on $\rzeczywiste^n$ and suppose that $\mu$ vanishes on Borel subsets of $\rzeczywiste^n$ of Hausdorff dimension at most $n-1$. Then there exists a convex function $\psi$ on $\rzeczywiste^n$ whose subgradient $\partial\psi$ pushes $\mu$ forward to $\nu$. $\partial\psi$ is uniquely determined $\mu$-almost everywhere.
\end{Thm}

Of course the assumption on the null sets of $\mu$ can not be abandoned. For example if $\mu = \delta_x$ and $\nu$ is not a point measure, then if $A$ is such a set that $0 < \nu[A] < 1$ one gets that for any convex function $\phi$, $\nu[A] \neq \mu[(\partial\phi)^{-1}(A)]$ since the latter must always be either 0 or 1.

\subsection{Rokhlin's disintegration theorem}

Suppose $Z$ is a compact metric space equipped with a Borel probability measure $\mu$. Let $\mathcal{P}$ be a partition of $Z$ into measurable sets, it is then straightforward to define a measure space structure on $\mathcal{P}$. Namely if $\pi:Z \rightarrow \mathcal{P}$ is a projection sending a point to the (unique) partition set containing it, then we define a probability measure $\widehat{\mu}$, by simply averaging $\mu$ over the preimages of $\pi$
\[
 \widehat{\mu}(\mathcal{Q}) = \int_{\pi^{-1}(\mathcal{Q})}d\mu.
\]

\begin{Def}
 Disintegration of $\mu$ with respect to $\mathcal{P}$ or a system of conditional measures of $\mu$ with respect to $\mathcal{P}$ is a family of probability measures $\eta_p$ for $p \in \mathcal{P}$ such that
 \begin{enumerate} \vspace{-1em}
  \item $\eta_p(p) = 1$ for $\widehat{\mu}$-a.e. $p \in \mathcal{P}$,
  \item for every continuous function $\displaystyle f:Z \rightarrow \rzeczywiste$ the function $\displaystyle \mathcal{P} \ni p \rightarrow \int_p f\,d\eta_p$ is measurable and $\displaystyle \int f\,d\mu = \int\left(\int_p f\,d\eta_p\right)\,d\widehat{\mu}(p)$.
 \end{enumerate}
\end{Def}

The question for which spaces and their partitions the systems of conditional measures exists arises immediately. For that we need one more definition.

\begin{Def}
 A partition $\mathcal{P}$ of $Z$ is called measurable if there exists a countable family of $\{E_i\} \subseteq 2^Z$ and a subset $W \subseteq Z$ of full measure such that for each $q \in \mathcal{P}$
 \[
  q \cap W = E^*_1 \cap E^*_2 \cap \ldots \cap W,
 \]
 where $E^*_j$ is either $E_j$ or $Z\setminus E_j$.
\end{Def}

Equipped with that notion we can state the theorem of Rokhlin~\cite{Rokhlin}

\begin{Thm}
 If $\mathcal{P}$ is a measurable partition of $(Z,\mu)$ then $\mu$ disintegrates with respect to $\mathcal{P}$. Moreover the conditional measures are unique $\widehat{\mu}$-a.e. 
\end{Thm}

\section{Proofs and calculations}\label{sec:wyniki}

\paragraph{The induced form on the level orbifolds.} Since the results are local let us fix some open convex chamber $\Delta$ and some local trivialisation of $\moment^{-1}(\Delta)$ in the $(x,y,w)$ variables mentioned before. Let $\phi$ be the local potential of the \Kahler from, as we mentioned it is convex in $x$. It will be worthwhile to study the relationship between the potential and its Legendre transform defined by
\[
 \phi^*(p,w) = \sup_{x \in \rzeczywiste^k}\{x\cdot p - \phi(x,w)\}.
\]
This is not a function on $W \times \Delta$ since it is not globally defined. On a local atlas of $W$ $\phi^*$ is convex in $p \in \Delta$ with fixed $w \in W$ and plurisuperharmonic in $w$ for fixed $p$~\cite{Kiselman}. If $\phi^*$ happens to be differentiable in $p$ its derivative is well defined globally and as such is an inverse of $\moment^+$ between $\left(\moment^{-1}(\Delta) \cap X^0\right)/\mathbf{T}^k$ and $\Delta^0 \times W$.

In fact on each orbit $O$ of $\torus_\zespolone^k$ in $X^1$ the momentum map $\moment$ will be a diffeomorphism of $O/\torus^k$ to $\Delta^0$. Indeed, by the definition of $X^1$, \eqref{fact} and the constant rank theorem we get that $\moment|_O$ is a local diffemorphism. That it is a global diffemorphism follows from the fact that $\torus^k_\zespolone$ act transitively on $O$ and the function $\moment^Y$ is strictly increasing along the flow of $JY^\#$. Note that by the properties of the Legendre transform ~\cite[2.1.3]{Villani} the inverse of $\moment|_O$ will be exactly $\nabla_p\phi^*(\cdot,w)$, with $w$ the parameter of the orbit $O$ in $W$. Hence, from this and the Implicit Function Theorem we imply the following identities
\begin{align*}
 & \nabla_x\phi(\nabla_p \phi^*(c,w),w) = c \\
 & \nabla_p\phi^*_{w_i} = -(D^2_x\phi)^{-1}\nabla_x\phi_{w_i}.
\end{align*}
Using those we can compute the local expression for the reduced form with little effort, namely 
\begin{Prop}
 For any regular value $q$ the reduced symplectic form can be locally expressed by
  \[
    \sigma_\phi(q) = - \frac{i}{2}\sum \phi^*_{w_i\bar{w}_j}(q,w)dw_i \wedge d\bar{w}_j.
  \]
\end{Prop}
\begin{proof}
Indeed, on the one hand
\[
 \phi^*_w(p,w) = -\phi_w(\phi^*_p(p,w),w).
\]
It holds since for any convex function $u$, $u(x) + u^*(p) = x\cdot p$ iff $p \in \partial u(x)$ iff $x \in \partial u^*(p)$. Taking another derivative we get
\[
 \phi^*_{w_i\bar{w}_j} = -\phi_{w_i\bar{w}_j} - \nabla_x\phi_{\bar{w}_j}\cdot \nabla_p\phi^*_{w_i} = -\phi_{w_i\bar{w}_j} + \nabla_x\phi_{\bar{w}_j}\cdot (D^2_x\phi)^{-1}\nabla_x\phi_{w_i}.
\]
On the other hand we can parametrize the level set by $\{x = \phi^*(q,w)\}$, with $q$ being constant, that allows us in turn to directly compute the form $i^*\omega$ using the local expressions for $\omega$ form Section~\ref{sec:torus}. Finally, after noticing that all summands involving $dy$ disappear we get the desired formula.
\end{proof}
It is a good place to note here that the Kirwan map mentioned above reduces in this coordinates to taking the Legendre transform. Indeed, suppose that $\omega$ has a local potential $f$. Then the local potential on $\Delta \times W$ will be the Legendre transform $f^*$ as follows from the above computations. Moreover, if $u$ is a global quasiplurisubharmonic function then there exists $\tilde{u}$ such that locally
\[
 f + u \longrightarrow (f+u)^* = f^* + \tilde{u}
\]
will be the potential for the form $\sigma_u$ reduced to $W$ from $\omega + i\partial\bar{\partial}u$ at some fixed regular value.

Let us now try to derive the formula for $\det\phi^*_{w\bar{w}}$, which at any fixed $q$ will be the local density of the symplectic volume of the quotient orbifold. 

\begin{Thm}
In the setting as above, the following relationship between volume densities holds
\[
 \det D^2\phi(x,w) = (-1)^m\det D^2_w\phi^*(\nabla_x\phi(x,w),w) \det D_x^2\phi(x,w),
\]
with $D_x$ and $D_w$ understood as restricting the derivation to toric and transveral variables respectively.
\end{Thm}
\begin{proof}
First let us notice that $\det D_x^2\phi(x,w)$ will be nonzero by the discussion at the beginning of this section and hence will always be invertible. Let us now prove the formula. Since we don't see any clever way to arrive at it we will just calculate. We also would like to recollect a well known fact that the inverse of an invertible matrix $A$ can be represented by formula $A^{-1} = C^T/\det(A)$ where $C$ is the cofactor matrix of $A$, namely a matrix whose $(i,j)$-th entry is the determinant of $A$ with $i$-th row and $j$-th column removed multiplied by $(-1)^{i+j}$. Since we are dealing with symmetric matrices, we can forget the transpose in the formula.

Let us denote the Hessian matrix of $x$ variables as $H$, by $D$ its determinant and by $D^{\alpha}_{\beta}$ the minor of $H$ with row $\alpha$ and column $\beta$ removed. Similarly we define $D^{\alpha_1,\ldots,\alpha_k}_{\beta_1,\ldots,\beta_k}$. The minor of the matrix with all rows and columns removed is defined to be 1. To make the notation a little more readable we will denote real derivatives by Greek letters and complex ones by Latin letters.

Let us look the determinant
\[
 \det(\sigma) = \sum_{t\in S} sgn(t)\prod_{i}\phi^{*}_{i,\bar{t(i)}} = \sum_{t\in S} sgn(t)\prod_{i}(-\phi_{i,t(i)} +\nabla_x\phi_i\cdot H^{-1}\cdot \nabla_x\phi_{\bar{t(i)}}) 
\]
Now we expand each product and look at the number of terms with $H^{-1}$ in each factor. Let us denote this number by $l$. If $l = 0$ then we get a summand
\[
 (-1)^m \det (\phi_{w\bar{w}}).
\]
If $l=1$ we get the following summands
\[
 \frac{(-1)^m}{D}\text{sgn}(t)\sum_{i, \alpha,\beta} (-1)^{\alpha+\beta-1}\phi_{i,\alpha} D^{\alpha}_\beta \phi_{\beta,\bar{t(i)}}\prod_{j\neq i}\phi_{j,\overline{t(j)}}
\]
and save for division by $D$ each of those appears in $(-1)^m\det(D^2\phi)$ with the exact sign $\text{sgn}(t)(-1)^{\alpha+\beta-1}$ since it's one addtional transposition from $\phi_{i,\bar{t(i)}}\phi_{\beta,\alpha}$.

Now suppose that $l >1$ and denote by $L$ a multi-index of length $l$ and by $L^c$ its completion of length $m-l$. For each pair of those we get the following summands
\[
 (-1)^{m}\prod_{k\in L}\left(\sum_{\alpha,\beta} (-1)^{\alpha+\beta-l}\phi_{k,\alpha} D^{\alpha}_\beta \phi_{\beta,\bar{t(k)}}\right)\prod_{i \in L^c}\phi_{i,\overline{t(i)}}
\]
for some fixed permutation $t$. Comparing those between permutations we notice that products in which $\alpha$ or $\beta$ appear at least twice cancel each other. Indeed, if determinants $D^{\alpha}_{\cdot}$ appear next to e.g. $\phi_{j,x_\alpha}$ and $\phi_{k,x_\alpha}$, then they will appear with opposite sign in summands with $t(j)$ and $t(k)$ reversed. For that reason, if $m > k$, then the products for $l > k$ will cancel out since some $\alpha$ and $\beta$ will have repeat itself in each product.

Armed with that observation we fix some $w$ indices $K$ and $\overline{w}$ indices $\overline{K}$ with fixed order and factor out the product $\prod_{i \in K}\phi_{i,\overline{t(i)}}$ ($t$ being the bijection from $j$-th element of $K$ to $j$-th element of $\overline{K}$). The remaining derivatives of $\phi$ will range over all the permutations of $\overline{K}^c$. Thus in front of $\prod_{i \in K}\phi_{i,\overline{t(i)}}$ there will be a following multiplier
\begin{align*}
 (-1)^m\sum_{u \in S(\overline{K}^c)}\text{sgn}(t+u) &\sum_{\alpha_l,\beta_l} (-1)^{|\alpha| + |\beta| -l^2}\sum_{r\in S(\alpha_L)}\text{sgn}(r) \quad  \\
  \quad & \left( \prod_{i\in K^c}\phi_{i,r(\alpha_i)}\phi_{\beta_i\overline{u(i)}}\right)\left(D^{-l}\sum_{s \in S_k}\text{sgn}(s) \prod_{i=1}^l D^{\alpha_i}_{\beta_{s(i)}}\right).
\end{align*}
Some clarification of the notation is due. By $\text{sgn}(t+s)$ we understand the sign of permutation that comes from concatenating $t$ and $u$, that is the permutation that applies $t$ to indices from $K$ and $u$ to indices from $K^c$. The symbols $\alpha_l$ and $\beta_l$ denote multiindices of length $l$, with $|\alpha| = \alpha_1 + \ldots + \alpha_l$.

Finally, the Proposition~\ref{Prop:iloczyny-wyznacznikow} proven at the end of the paper implies that this multiplier is in fact equal to
\[
 \frac{(-1)^m}{D} \hspace{-4pt} \sum_{u \in S(\overline{K}^c)}\text{sgn}(t+u)\sum_{\alpha_l,\beta_l} (-1)^{|\alpha| + |\beta| -l^2} \hspace{-5pt} \sum_{r\in S(\alpha_L)} \hspace{-3pt} \text{sgn}(r) \hspace{-3pt} \left( \prod_{i\in K^c} \hspace{-1pt} \phi_{i,r(\alpha_i)}\phi_{\beta_i\overline{u(i)}}\right) \hspace{-1pt} \left( D^{\alpha_1,\ldots,\alpha_k}_{\beta_1,\ldots,\beta_k}\right) \hspace{-1pt} .
\]
Each of those terms will appear exactly once in $(-1)^m \det (\phi_{w\bar{w}})/D$. Indeed, applying the Laplace expansion to $D$ from the highest index in $\alpha_l$ downward we will eventually arrive at the term $(-1)^{|\alpha| + |\beta|}\prod\phi_{\beta_i,\alpha_i}D^{\alpha_1,\ldots,\alpha_k}_{\beta_1,\ldots,\beta_k}$. It will be multiplied by the term $\text{sgn}(s)\prod \phi_{i,\overline{s(i)}}$ with $s = t + u$, understood as above. Now we switch $i$ with $\beta_i$ for each $i \in K^c$ and that will result with additional $(-1)^l = (-1)^{l^2}$ factor, and finally permute $\alpha_l$ to get additional multiplier $\text{sgn}(r)$.

Finally, we couclude that
\[
 \det(\sigma(p,w)) = (-1)^m\frac{\det D^2\phi(\nabla_p\phi^*(p,w),w)}{\det D^2_x \phi(\nabla_p\phi^*(p,w),w)}
\]
and the claim follows from the fact that $p = \nabla_x\phi(x,w)$ for some $x$.
\end{proof}

Armed with the above formula we can prove the first theorem. Let us take a measure $\mu$ and a solution $u$ to $g$-\Monge problem. Since the complement of $X^1$ is a pluripolar set, we can safely restrict our considerations to $X^1$. It is a set of full measure that is also a principal fibration over a compact orbifold (which among other things is second countable and locally compact), so the Rokhlin's theorem clearly applies in this setting. Let measures $\eta_w(x)\, dx$ be the conditional measures on orbits with the average $\hat{\mu}$. The conditionals are unique $\hat{\mu}$-a.e. Moreover, since the base of the fiber bundle $W$ is in a sense a space of fibers, the average measure can be naturally understood as a measure on $W$.

The formula for the determinant of the reduced form allows us to define the conditional measures immediately as
\[
 \eta_w(x) = 
  (-1)^m\frac{g(\moment_u) \sigma^m(\moment_u,w) }{\hat{\mu}(w)}\det D_x^2 u(x,w).
\]
With $\hat{\mu}(w)$ being simply the local density of $\hat{\mu}$ and serving as a normalizing factor.

To prove the conditional statement about the uniqueness of solutions to the transport \Monge equation let us first mention a result that will be of use in the proof

\begin{Thm}[\cite{Faulk}]
 Let $(W,\sigma)$ be a compact \Kahler orbifold, then for any smooth $f$ the equation
 \[
  (\sigma + i\partial\bar{\partial}\phi)^n = e^f\sigma^n
 \]
 always has a unique solution, provided both measures have the same total mass.
\end{Thm}

Since we have a formula for the conditional measures, given a family of reduced forms parametrized by $\Delta$, instead of asking for a solution to $g$-\Monge problem we can ask for a family of solutions, one for each orbit that satisfies the formula proved above. That is the content of the Proposition below. 

\begin{Prop}
 Suppose $u$ and $v$ are two solutions to the $g$-\Monge problem. Then then following statements are equivalent:
 \begin{enumerate} \vspace{-0.5em}
  \item The conditional measures have $\hat{\mu}$-a.e. unique solutions orbit-wise, i.e. on each orbit form the set of $\hat{\mu}$-full measure of $W$ there is only one (up to additive constant) convex function $F$ that satisfies
  \[
   \eta_w(x) = (-1)^m\frac{g(\nabla_x F) \sigma_u^m(\nabla_x F,w) }{\hat{\mu}(w)}\det D_x^2 F(x).
  \]

  \item $\sigma_u = \sigma_v$, $\hat{\mu}$-a.e..
 \end{enumerate}
 Moreover, any of the above statement implies global uniqueness of solutions.
\end{Prop}
\begin{proof}
 Suppose the second assertion holds. Then at $\hat\mu$-almost every orbit the potential restricted to this orbit solves the equation
 \[
  \frac{1}{\hat{\mu}(w)}g(\nabla u) \sigma^m(\nabla u, w)\det D^2_x u = \eta_w(x),
 \]
 with $w$ serving as an orbit parameter. By McCann's result the solution to this problem is unique up to an additive constant, thus the difference between $u$ and $v$ can only depend on $w$. 
 
 Suppose the first assertion holds. Thus the moment map restricted to the orbit serves as an optimal transport map between $\eta_w$ and $g\sigma^m_u$ (disregarding the normalization). Since those maps are unique, that implies that the target measures must be equal, which in turn implies that $\sigma^m_u = \sigma^m_v, \hat{\mu}$-a.e..
 
 Suppose there is a set $A \subseteq W$, such that $\hat{\mu}[A] = 0$ and on this set $\sigma^m_u(p,A) \neq \sigma^m_v(p,A)$ for some $p$. If the $p$'s for which this holds form a set of non-zero measure then taking the average over $\Delta$ one gets that $\hat{\mu}[A] \neq 0$, thus for a.e. $p$ we must have that $\sigma^m_u = \sigma^m_v$ a.e.. For those $p$'s we can solve the \Monge equation on the orbifold $W$. Indeed, recall that going from the manifold to the reduced space induces a well defined map in cohomology, thus the potentials corresponding to $u$ and $v$ will be the potentials in the same cohomology class. Let us denote them by $u^*$ and $v^*$. By the solution of \Monge equation on orbifolds with right-hand $\sigma^m_v$ we get that $\sigma_u = \sigma_v$ and $u^* - v^* = c$. Moreover, $W \times \Delta$ being compact implies that $u^* - v^*$ differ only by a function of $p$, say $f(p)$. Together with the first statement it implies $f(p) = C$.
\end{proof}

The main result follows easily as corollary from the above theorem

\begin{Cor}
 If two solutions provide the same momentum map they must differ by a constant.
\end{Cor}
\begin{proof}
 The proof is trivial - the measure is determined by the conditionals and the average, thus both functions must satisfy the first point of the above proposition with the same $\sigma^m$, since the optimal transport maps are unique.
\end{proof}

\section{Proof of formula~\ref{Prop:iloczyny-wyznacznikow}} 

The only thing left to prove is the following

\begin{Prop}\label{Prop:iloczyny-wyznacznikow} Suppose we are given a $n \times n$ matrix $M$. Let us denote the determinant of the matrix by $D$, and by $D^\alpha_\beta$ the determinant of the matrix with $\alpha$-th row and $\beta$-th column removed. Then for  every $k \in \{2,\ldots,n\}$ and for every multi-index $(\alpha_1,\ldots,\alpha_k)$ with $\alpha_1 < \ldots < \alpha_k$ and $(\beta_1,\ldots,\beta_k)$ with $\beta_1 < \ldots < \beta_k$ the following equality holds
\[
 \sum_{s \in S_k}\text{sgn}(s) \prod_{i=1}^k D^{\alpha_{s(i)}}_{\beta_i} = D^{k-1}D^{\alpha_1,\ldots,\alpha_k}_{\beta_1,\ldots,\beta_k}
\]

\end{Prop}

\begin{Rem}
 The formula goes back probably to Muir. In~\cite{DeterminantalIdentities} one can see the proof of sort of `dual` identity, where the rows are added instead of removed. As before, we define the determinant of the matrix with all rows and columns removed as 1.
\end{Rem}

\begin{proof}
The proof goes by induction on $n$ and $k$. For $n=2$ and $k=2$ the formula is simply the definition of $2 \times 2$ determinant. Generally, for $n = k$ the formula is just the definition of determinant of the cofactor matrix of $M$, which is indeed $D^{k-1}$.

Before we start the proof there is one more notational remark: we will often use summation over some symmetric group of incomplete set of indices. For example, suppose the lower set of indices misses some $j$ and upper misses $i$. By permutation of such set we understand renumbering both sets and then permuting. With that out of the way we are ready to prove.

Assume now that the formula holds for some $k-1$ and for $k$ up to $n-1$, we would like to prove the formula for $n$ and $k$.

First, let us deal with the left-hand side. We may assume that for $\alpha_1$ there is at least one $\beta_i$ such that $D^{\alpha_1}_{\beta_i} \neq 0$. Indeed, if not the proof becomes trivial since then the left hand-side vanishes and elementary proof shows that the right-hand side must vanish too. Thus we may assume that there is some $i$ such that $D^{\alpha_1}_{\beta_i} \neq 0$. Since it does not effect the proof's mechanics and makes sign computation little less tiresome we will assume that $i=1$. Then we have that
\begin{align*}
 \sum_{s \in S_k}\text{sgn}(s)& \prod_{i=1}^k D^{\alpha_{s(i)}}_{\beta_i} = \sum_i (-1)^{i-1}D^{\alpha_{i}}_{\beta_1}\left(\sum_{t \in S_{k-1}}\text{sgn}(t)\prod_{j > 1}D^{\alpha_{t(j)}}_{\beta_j}\right) \\
 & =\sum_i (-1)^{i-1}D^{\alpha_{i}}_{\beta_1}\left(D^{k-2} D^{\alpha \setminus \alpha_i}_{\beta_{2,\ldots,k}}\right),
\end{align*}
where the first equality comes from factoring out the number of inversions involving $i$ and the second comes from the induction hypothesis on $k-1$. Now assume that there is another $j$ such that $D^{\alpha_1}_{\beta_j} \neq 0$. If there is not, then again the proof simplifies since we are left with the following claim
\[
 D^{\alpha_1}_{\beta_1}\left(D^{k-2}D^{\alpha_2,\ldots,\alpha_k}_{\beta_2,\ldots,\beta_k}\right) = D^{k-1}D^{\alpha_1,\ldots,\alpha_k}_{\beta_1,\ldots,\beta_k}
\]
and by the assumption on the nonvanishing of $D^{\alpha_1}_{\beta_1}$ and induction assumptions
\begin{align*}
 D^{k-1}D^{\alpha_1,\ldots,\alpha_k}_{\beta_1,\ldots,\beta_k} &= D^{k-1}\frac{\sum_{t\in S_{k-2}}\text{sgn}(t)\prod_j D^{\alpha_1,\alpha_{t(j)}}_{\beta_1,\beta_j}}{(D^{\alpha_1}_{\beta_1})^{k-3}} \\
 &= D\frac{\sum_{t\in S_{k-2}}\text{sgn}(t)\prod_j DD^{\alpha_1,\alpha_{t(j)}}_{\beta_1,\beta_j}}{(D^{\alpha_1}_{\beta_1})^{k-3}} \\
 &= D\frac{\sum_{t\in S_{k-2}}\text{sgn}(t)\prod_j D^{\alpha_1}_{\beta_1}D^{\alpha_{t(j)}}_{\beta_j} - D^{\alpha_1}_{\beta_j}D^{\alpha_{t(j)}}_{\beta_1}}{(D^{\alpha_1}_{\beta_1})^{k-3}} \\
 &= D D^{\alpha_1}_{\beta_1}\sum_{t\in S_{k-2}}\text{sgn}(t)\prod_jD^{\alpha_{t(j)}}_{\beta_j} = D^{\alpha_1}_{\beta_1}\left(D^{k-2}D^{\alpha_2,\ldots,\alpha_k}_{\beta_2,\ldots,\beta_k}\right).
\end{align*}
Thus we might suppose that there is an index $j > 1$ such that $D^{\alpha_1}_{\beta_j} \neq 0$. For the same reason as before we will take $j = k$. Finally, the left hand-side equals
\[
 D^{\alpha_1}_{\beta_1}\left(D^{k-2}D^{\alpha_2,\ldots,\alpha_k}_{\beta_2,\ldots,\beta_k}\right) + \sum_{i>1} (-1)^{i-1}D^{\alpha_{i}}_{\beta_1}\left(D^{k-2}\frac{\sum_{t\in S_{k-2}}\text{sgn}(t)\prod_j D^{\alpha_1,\alpha_{t(j)}}_{\beta_k,\beta_j}}{(D^{\alpha_1}_{\beta_k})^{k-3}}\right).
\]

With the right-hand side we use the same tricks
\[
 D^{k-1}D^{\alpha_1,\ldots,\alpha_k}_{\beta_1,\ldots,\beta_k} 
 = 
 D^{k-1}\frac{\sum_{t\in S_{k-1}}\text{sgn}(t)\prod_j D^{\alpha_1,\alpha_{t(j)}}_{\beta_k,\beta_j}}{(D^{\alpha_1}_{\beta_k})^{k-2}}.
\]
In each product above there is an element with second pair $\phantom{.}^{\alpha_{t(1)}}_{\beta_1}$, we factor it out together with $(-1)^{t(i)}$, which is the number of inversions for this element, since $\alpha_1$ is not counted. As before we multiply this element by one of $D$'s and use hypothesis for $n=2$ to get
\[
 D^{k-2}\sum_i(-1)^i\left(D^{\alpha_1}_{\beta_1}D^{\alpha_i}_{\beta_k} - D^{\alpha_i}_{\beta_1}D^{\alpha_1}_{\beta_k}\right)\frac{\sum_{t\in S_{k-2}}\text{sgn}(t)\prod_j D^{\alpha_1,\alpha_{t(j)}}_{\beta_k,\beta_j}}{(D^{\alpha_1}_{\beta_k})^{k-2}}.
\]
We notice that the terms with $D^{\alpha_i}_{\beta_1}D^{\alpha_1}_{\beta_k}$ are exactly those from the left-hand side sum except the first one, thus we are left with
\begin{align*}
 &D^{k-2}\sum_i(-1)^i D^{\alpha_1}_{\beta_1}D^{\alpha_i}_{\beta_k}\frac{\sum_{t\in S_{k-2}}\text{sgn}(t)\prod_j D^{\alpha_1,\alpha_{t(j)}}_{\beta_k,\beta_j}}{(D^{\alpha_1}_{\beta_k})^{k-2}} \\
 =\; &D^{\alpha_1}_{\beta_1}\left( \sum_i(-1)^i D^{\alpha_i}_{\beta_k}\frac{\sum_{t\in S_{k-2}}\text{sgn}(t)\prod_j D D^{\alpha_1,\alpha_{t(j)}}_{\beta_k,\beta_j}}{(D^{\alpha_1}_{\beta_k})^{k-2}}\right) \\
 =\; &D^{\alpha_1}_{\beta_1}\left( \sum_i(-1)^i D^{\alpha_i}_{\beta_k}\frac{\sum_{t\in S_{k-2}}\text{sgn}(t)\prod_j\left( D^{\alpha_1}_{\beta_j}D^{\alpha_{t(j)}}_{\beta_k} - D^{\alpha_1}_{\beta_k}D^{\alpha_{t(j)}}_{\beta_j}\right)}{(D^{\alpha_1}_{\beta_k})^{k-2}}\right).
\end{align*}
After multiplying out the differences the final observation is the following: for each permutation the only term that will not be canceled is the one involving the term $(D^{\alpha_1}_{\beta_k})^{k-2}$. Indeed, suppose we have a product with $\alpha_i$ removed from possible permutations and a factor $D^{\alpha_1}_{\beta_j}D^{\alpha_{t(j)}}_{\beta_k}$ in this product, then there will be the same product with $t(j)$ and $i$ reversed. We can suppose without losing any generality that $t(j) > i$. The signs of those two permutations will differ by $t(j)-i -1$ which is exactly the number of inversions that will either appear or disappear by changing $i$ to $t(j)$, thus both products will cancel and we are left with
\[
 D^{\alpha_1}_{\beta_1}\left( \sum_i(-1)^{i+k-2} D^{\alpha_i}_{\beta_k}\sum_{t\in S_{k-2}}\text{sgn}(t)\prod_j D^{\alpha_{t(j)}}_{\beta_j}\right).
\]
But then obviously $(-1)^{k+i-2} = (-1)^{(k-1)-(i-1)}$ and by induction hypothesis on $k$ we end up with
\[
 D^{\alpha_1}_{\beta_1}\left(D^{k-2}D^{\alpha_2,\ldots,\alpha_k}_{\beta_2,\ldots,\beta_k}\right).
\]

To finish the proof we must prove the formula for $k=2$. Assume then that it holds for some $n-1$. Let $\alpha < \gamma$ and $\beta < \delta$,  then
\begin{align*}
 D^2 D^{\alpha,\gamma}_{\beta,\delta} =& \sum_i (-1)^{\alpha+i}m_{\alpha,i}\,D^{\alpha}_i D^{\alpha,\gamma}_{\beta,\delta} = (-1)^{\alpha+\beta}m_{\alpha,\beta}\,D^{\alpha}_\beta D^{\alpha,\gamma}_{\beta,\delta}\\
 + &\sum_{i<\beta} (-1)^{\alpha+i}m_{\alpha,i}\left(\sum_{j<\alpha}(-1)^{\beta+j+1}m_{j,\beta}D^{\alpha,j}_{\beta,i} + \sum_{j>\alpha}(-1)^{\beta+j}m_{j,\beta}D^{\alpha,j}_{\beta,i}\right)D^{\alpha,\gamma}_{\beta,\delta} \\
 + &\sum_{i>\beta} (-1)^{\alpha+i}m_{\alpha,i}\left(\sum_{j<\alpha}(-1)^{\beta+j}m_{j,\beta}D^{\alpha,j}_{\beta,i} + \sum_{j>\alpha}(-1)^{\beta+j+1}m_{j,\beta}D^{\alpha,j}_{\beta,i}\right)D^{\alpha,\gamma}_{\beta,\delta}
\end{align*}
just by using the Laplace expansion twice. Now by the induction hypothesis we have
\[
 D^{\alpha,j}_{\beta,i}D^{\alpha,\gamma}_{\beta,\delta} = (-1)^{\text{sgn}(\gamma - j)\text{sgn}(\delta-i)}D^{\alpha}_{\beta}D^{\alpha,\gamma,j}_{\beta,\delta,i} + D^{\alpha,\gamma}_{\beta,i}D^{\alpha,j}_{\beta,\delta}
\]
and we plug that into the sum above.

On the other hand we have 
\begin{align*}
 D^{\alpha}_{\beta}D^{\gamma}_{\delta} &= D^{\alpha}_{\beta}\left(\sum_{i<\delta} (-1)^{\alpha+i}m_{\alpha,i}D^{\alpha,\gamma}_{\delta,i} +  \sum_{i > \delta} (-1)^{\alpha+i + 1}m_{\alpha,i}D^{\alpha,\gamma}_{\delta,i}\right) =\\
 &(-1)^{\alpha+\beta}m_{\alpha,\beta}\,D^{\alpha}_{\beta} D^{\alpha,\gamma}_{\beta,\delta} +
 \end{align*}
 \begin{align*}
 \vphantom{=} D^{\alpha}_{\beta}\left(\sum_{i<\beta} (-1)^{\alpha+i}m_{\alpha,i}\left(\sum_{j<\alpha}(-1)^{\beta+j+1}m_{j,\beta}D^{\alpha,\gamma,j}_{\delta,\beta,i} \right.\right. + &\sum_{\alpha < j < \gamma}(-1)^{\beta+j}m_{j,\beta}D^{\alpha,\gamma,j}_{\beta,\delta,i} \\
 &+ \left. \sum_{j > \gamma}(-1)^{\beta+j+1}m_{j,\beta}D^{\alpha,\gamma,j}_{\beta,\delta,i}\right) 
 \end{align*}
 \begin{align*}
 \vphantom{= D^{\alpha}_{\beta}\left(\right.}+ \sum_{\beta<i<\delta} (-1)^{\alpha+i}m_{\alpha,i}\left(\sum_{j<\alpha}(-1)^{\beta+j}m_{j,\beta}D^{\alpha,\gamma,j}_{\delta,\beta,i} \right. + &\sum_{\alpha < j < \gamma}(-1)^{\beta+j+1}m_{j,\beta}D^{\alpha,\gamma,j}_{\beta,\delta,i}  \\
 &+ \left.\sum_{j > \gamma}(-1)^{\beta+j}m_{j,\beta}D^{\alpha,\gamma,j}_{\beta,\delta,i}\right)
 \end{align*}
 \begin{align*}
 \vphantom{= D^{\alpha}_{\beta}\left(\right.}+ \sum_{i>\delta} (-1)^{\alpha+i+1}m_{\alpha,i}\left(\sum_{j<\alpha}(-1)^{\beta+j}m_{j,\beta}D^{\alpha,\gamma,j}_{\delta,\beta,i} \right. + &\sum_{\alpha < j < \gamma}(-1)^{\beta+j+1}m_{j,\beta}D^{\alpha,\gamma,j}_{\beta,\delta,i}\\
 &+ \left.\left.\sum_{j > \gamma}(-1)^{\beta+j}m_{j,\beta}D^{\alpha,\gamma,j}_{\beta,\delta,i}\right)\right)
\end{align*}
and for the other term
\begin{align*}
 -D^{\gamma}_{\beta}D^{\alpha}_{\delta} =& \\
 -&\left(\sum_{i<\beta} (-1)^{\alpha+i}m_{\alpha,i}D^{\alpha,\gamma}_{\beta,i} +  \sum_{i > \beta} (-1)^{\alpha+i + 1}m_{\alpha,i}D^{\alpha,\gamma}_{\beta,i}\right) \times\\
 &\left(\sum_{j<\alpha} (-1)^{\beta+j}m_{j,\beta}D^{\alpha,j}_{\delta,\beta} +  \sum_{j > \alpha} (-1)^{\beta+j+ 1}m_{\alpha,i}D^{\alpha,\gamma}_{\delta,i}\right).
\end{align*}
Now after staring at both sides long enough one should see that they are in fact equal.
\end{proof}

\subsection*{Acknowledgement}
The author would like to thank S\l{}awomir Dinew for his encouragement. The author was supported by Polish National Science Centre grant 2018/29/N/ST1/02817. 

\bibliographystyle{acm}
\bibliography{/afs/matinf.uj.edu.pl/usr/dokt/im/myga/k.../ooo/ooo/sanina/Bibliografia/bibliografia_toryczna}
\vspace{1em}
\textsc{Szymon Myga, Department of Mathematics and Computer Science, Jagiellonian University, Poland}

\textit{E-mail address: }\texttt{\email}
\end{document}